\documentclass[11pt]{article}

\usepackage{amsmath,amsthm,amsfonts,amssymb,mathrsfs,epsfig,bm}

\renewcommand{\Re}{{\rm I\kern-0.16em R}}

\def\@begintheorem#1#2{\trivlist \item[\hskip \labelsep{\bf #1\ #2}]}
\def\@opargbegintheorem#1#2#3{\trivlist
      \item[\hskip \labelsep{\bf #1\ #2\ (#3)}]}

\newtheorem{proposition}{Proposition}[section] 

\newtheorem{thm}[proposition]{Theorem}
\newtheorem{lemma}[proposition]{Lemma}
\newtheorem{corollary}[proposition]{Corollary}
\newtheorem{example}[proposition]{Example}
\newtheorem{remark}[proposition]{Remark}

\def\e{\hbox {\rm e}}

\def\R{{\bf R}}
\def\R{{\bf R}}

\def\E{{\bf E}}

\def\ri{\right}
\def\le{\left}

\begin{document}

\author{
Paavo Salminen\\{\small \AA bo Akademi University,}
	\\{\small Faculty of Science and Engineering,}
	\\{\small FIN-20500 \AA bo, Finland,} 
	\\{\small \tt phsalmin@abo.fi}
\and
Lioudmila Vostrikova\\
	{\small  Universit\'e d'Angers,}\\
	{\small  D\'epartement de Math\'ematiques,}\\
        {\small F-49045 Angers Cedex 01, France,}\\
        {\small \tt vostrik@univ-angers.fr}
}
\vskip5cm

\title{
On moments of  integral exponential functionals of additive processes} 
\date{}

\maketitle

\begin{abstract} 
\noindent
For real-valued additive  process  $(X_t)_{t\geq 0}$ 
 a recursive equation is derived for the entire positive moments of functionals 
$$
I_{s,t}= \int _s^t\exp(-X_u)du, \quad 0\leq s<t\leq\infty, 
$$ 
in case  the Laplace exponent of $X_t$ exists for positive values of the parameter. 
From the equation emerges
an easy-to-apply sufficient condition for the finiteness of the  
moments. As an application we study  first hit
processes of  diffusions.
\\ \\


\end{abstract}
 \eject
\begin{section}{Introduction}\label{s1}

Let $X=(X_t)_{t\geq 0},\, X_0=0,$ be a real valued additive process, i.e., a strong Markov process with  independent increments having c\`adl\`ag sample paths which are continuous in probability  (cf. Sato  \cite{Sa}, p.3). Important examples of additive processes are:

\begin{description}
\item{(a)}\hskip2.5mm Deterministic time transformations of L\'evy processes, that is, if $(L_s)_{s\geq 0}$ is a L\'evy process and $s\mapsto g(s)$ is an increasing continuous function such that $g(0)=0$ then $(L_{g(s)})_{s\geq 0}$ is an additive process. 
\item{(b)}\hskip2.5mm Integrals of deterministic functions with respect to a L\'evy process, that is,  if $(L_s)_{s\geq 0}$ is a L\'evy process and $s\mapsto g(s)$ is a measurable and locally bounded  function then 
$$
Z_t:=\int_0^t g(s)\,dL_s,\quad t\geq 0,
$$
is an additive process.
\item{(c)}\hskip2.5mm First hit processes of one-dimensional diffusions, that is, if $(Y_s)_{s\geq 0}$ is a diffusion taking values in $[0,\infty),$ starting from 0, and drifting to $+\infty$ then 
$$
H_a:=\inf\{ t \, :\, Y_t>a\},\quad a\geq 0,
$$
is an additive process. 
\end{description}
 Of course, L\'evy processes themselves constitute a large and important class of additive processes.

The aim of this paper is  to study integral exponential 
functionals of $X,$ i.e., functionals of the form 
\begin{equation}\label{deft}
I_{s,t}:= \int _s^t\exp(-X_u)du, \quad 0\leq s<t\leq\infty, 
\end{equation}
in particular, the moments of $I_{s,t}.$  We refer also to a companion paper \cite{SV}, where stochastic calculus is used to 
study the Mellin transforms of $I_{s,t}$ when the underlying additive process is a semimartingale with absolutely continuous characteristics. 

The main result of the paper is a recursive equation, see \eqref{recur00} below, 
which generalizes   the formula for L\'evy processes presented in Urbanik \cite{U} 
and Carmona, Petit and Yor \cite{CPY}, 
see also Bertoin and Yor \cite{BY}. 
This formula for L\'evy processes is also displayed below in \eqref{by2}. In Epifani, Lijoi and Pr\"unster \cite{ELP} an extension of the L\'evy process formula to integral functionals up to $t=\infty$ of increasing additive processes is discussed,  and their formula (7) 
can be seen as  a special case of our formula \eqref{pm22} -- as we found out after finishing our work. 
We refer to these papers and also to \cite{SV} for further references and applications, e.g.,  in financial mathematics and statistics.  

In spite of the existing  closely related results we feel that it is worthwhile to provide a more thorough discussion of the topic. We also give new (to our best knowledge) applications of the formulas for first hit processes of diffusions and present explicit results for Bessel processes and geometric Brownian motions.

\end{section}

\begin{section}{Main results} 

Let $(X_t)_{t\geq 0}$ be an additive process and define
 for $0\leq s\leq t\leq\infty$ and $\alpha\geq 0$ 

\begin{equation}
\label{momn}
m^{(\alpha)}_{s,t}:=\E\left(I^\alpha_{s,t}\ri)=\E\left(\left(\int_s^t\e^{-X_u}\,du\right)^\alpha\ri), \quad
  \alpha\geq 0,
\end{equation}
and
$$
m^{(\alpha)}_{t}:=m^{(\alpha)}_{0,t},\qquad
m^{(\alpha)}_{\infty}:=m^{(\alpha)}_{0,\infty}.
$$
In this section we derive  a recursive integral equation for $m^{(\alpha)}_{s,t}$ under the following assumption:

\begin{description}
\item{(A)} {\it 
 \quad${\bf E}( {\rm e}^ {-\lambda X_t})<\infty$  for all $t\geq 0$ and  $\lambda\geq 0.$  
 }
\end{description} 
\noindent 
Under this assumption we define
\begin{equation}
\label{piilap}
\Phi(t;\lambda):=-\log {\bf E}( {\rm e}^ {-\lambda X_t}). 
\end{equation}
Since $X$ is assumed to be continuous in probability it follows that 
$t\mapsto \Phi(t,\lambda)$ is continuous. Moreover, $X_0=0$ a.s. implies that $\Phi(0;\lambda)=0.$  
If $X$ is a L\'evy process satisfying (A) we write (with a slight abuse of the notation)  
\begin{equation}
\label{piilap11}
{\bf E}( {\rm e}^ {- \lambda X_t}) = {\rm e}^{-t\Phi(\lambda)}.
\end{equation}
See Sato
\cite{Sa}  Theorem 9.8, p.52, for properties and the structure of the Laplace  exponent $\Phi$ of the infinitely divisible  distribution $X_t.$

In particular, (A) is valid for increasing additive processes.  Important examples of these are the first hit processes for diffusions (cf. (c) in Introduction). 
Assumption (A) holds also for additive processes of type (a) in Introduction  if the underlying L\'evy process fullfills (A). 

\begin{remark}
If $X$ is a semimartingale with absolutely continuous characteristics, a sufficient condition  for the existence of  the Laplace exponent as in (A)  in terms of the jump measure is given in Proposition 1 in  \cite{SV}.

\end{remark}

By continuity we have  -- from Jensen's inequality -- the following result  

\begin{lemma}\label{lemma1}
Under assumption  {\rm (A)}   $m^{(\alpha)}_{s,t}<\infty$ for all  $0\leq \alpha\leq 1$ and  $0\leq s\leq t<\infty.$ 
\end{lemma} 

The main result of the paper is given in the next theorem. In the proof we are using similar ideas as in \cite{CPY}. 
 
\begin{thm} 
\label{prop300}
Under assumption (A) it holds  for $ \alpha\geq 1$ and  $0\leq s\leq t< \infty$  that the moments $ m^{(\alpha)}_{s,t}$ are finite and satisfy the recursive equation
\begin{equation}
\label{recur00}
 m^{(\alpha)}_{s,t}=\alpha\int_s^t
 m^{(\alpha-1)}_{u,t}\e^{-\le(\Phi(u;\alpha)-\Phi(u;\alpha-1)\right)}\, du.
\end{equation}
\end{thm} 
\begin{proof} We start with by noting that the function $s:[0,t] \mapsto  I_{s,t}$ is for any $t>0$  continuous and strictly decreasing. Hence,  for $\alpha\geq 1$ 
$$
 \alpha\,\int_0^s I_{u,t}^{\,\alpha - 1} \, d I_{u,t}= \alpha\,\int_{I_{0,t}}^{I_{s,t}} v^{\alpha-1}\,dv= I_{s,t}^{\alpha} -I_{0,t}^{\alpha},
  $$
 where the integral is a pathwise (a.s.) Riemann-Stiltjes integral and the formula for the change of variables (see, e.g.,  Apostol \cite{A}, p. 144) is used. Consequently, from the definition of  $I_{s,t}$ it follows that 
$$
 I_{s,t}^{\alpha} -I_{0,t}^{\alpha} =
 -\alpha\,\int_0^s I_{u,t}^{\,\alpha - 1} \, {\rm e}^{-X_{u-}}\, du= -\alpha\,\int_0^s I_{u,t}^{\,\alpha - 1} \, {\rm e}^{-X_{u}}\, du.
  $$
Introducing the shifted functional $\widehat I_{s,t}$ via
$$
\widehat I_{s,t}:=\int_0^{t-s} {\rm e}^{-(X_{u+s}-X_s)}\, du
$$
we have 
\begin{equation}
\label{shiftedI}
\widehat I_{s,t}={\rm e}^{X_s}\, I_{s,t}={\rm e}^{X_s}\,\int_s^t {\rm e}^{-X_u}\, du,
\end{equation}
and, therefore, 
\begin{equation}
\label{fubini}
 I_{s,t}^{\alpha} -I_{0,t}^{\alpha} =
 -\alpha\,\int_0^s\widehat I_{u,t}^{\,\alpha - 1} \,{\rm e}^{-\alpha X_u}\, du.
\end{equation}
Notice that the independence of  increments  implies that  $\widehat I_{u,t}^{\,\alpha - 1} $ and ${\rm e}^{-\alpha X_u}$ are independent, and, hence, for all $\alpha\geq 1$
\begin{equation}
\label{mean4}
\E\le(\widehat I_{u,t}^{\, \alpha }\ri)= \E\le( I_{u,t}^{\alpha}\ri) /\E\le({\rm e}^{-\alpha X_u}\ri)
\end{equation}
Then evoking Lemma \ref{lemma1} and (\ref{mean4}) yield for $ 0\leq \alpha\leq 1$ and  $0\leq s\leq t<\infty$ 
\begin{equation}
\label{shiftedII} 
\E\big( \widehat I_{s,t}^{\alpha}\big)\leq \E\big( I_{0,t}^{\alpha}\big)/\E\le({\rm e}^{-\alpha X_u}\ri)<\infty.  
\end{equation}
Assume now that $\alpha\in[1,2]$. Taking the expectations in (\ref{fubini}) and applying Fubini's theorem gives  
\begin{equation}
\label{mean1}
\E\le( I_{s,t}^{\alpha} -I_{0,t}^{\alpha}\ri) =
 -\alpha\,\int_0^s\E\le(\widehat I_{u,t}^{\,\alpha - 1}\ri) \,\E\le({\rm e}^{-\alpha X_u}\ri)\, du >-\infty
\end{equation}
where the finiteness follows from (\ref{shiftedII}).
Since  $I_{s,t}\to 0$ a.s. when $s\uparrow t$ we obtain by applying monotone convergence in (\ref{mean1}) 
\begin{equation}
\label{mean2}
\E\le( I_{0,t}^{\alpha}\ri) =
 \alpha\,\int_0^t\E\le(\widehat I_{u,t}^{\,\alpha - 1}\ri) \,\E\le({\rm e}^{-\alpha X_u}\ri)\, du<\infty.
\end{equation}
Putting (\ref{mean1}) and   (\ref{mean2}) together results to the equation
\begin{equation}
\label{mean3}
\E\le( I_{s,t}^{\alpha}\ri) =
 \alpha\,\int_s^t\E\le(\widehat I_{u,t}^{\,\alpha - 1}\ri) \,\E\le({\rm e}^{-\alpha X_u}\ri)\, du.
\end{equation}
Finally, using (\ref{mean4}) and (\ref{shiftedII})  in  (\ref{mean3}) and recalling (\ref{piilap}) yield (\ref{recur00}) for $\alpha\in[1,2].$ Since (\ref{mean4}) is valid for all $\alpha$ and, as just proved, the finiteness holds for $\alpha\in[1,2]$ the proof of (\ref{recur00}) for arbitrary $\alpha>2$ is easily accomplished by induction. 
\end{proof} 

\begin{corollary} \label{corrlevy}
Let $(X_t)_{t\geq 0}$ be a L\'evy process with the Laplace exponent
as in (\ref{piilap11}). Then (\ref{recur00}) with  $s=0$ and $t<\infty$ 
is equivalent to
\begin{equation}
\label{recur2}
 m^{(\alpha)}_{t}=\alpha\,\e^{-t\Phi(\alpha)}\int_0^t
 m^{(\alpha-1)}_{u}\e^{u\Phi(\alpha)}\, du.
\end{equation}
\end{corollary}

\begin{proof} 
Put $s=0$ in (\ref{recur00}) to obtain
\begin{equation}
\label{recur3}
 m^{(\alpha)}_{t}=\alpha\int_0^t
 m^{(\alpha-1)}_{u,t}\e^{-u\le(\Phi(\alpha)-\Phi(\alpha-1)\right)}\, du.
\end{equation}
Consider
\begin{eqnarray*}
 m^{(\alpha-1)}_{u,t}&=& \E\left(\left(\int_u^t\e^{-X_v}\,dv\right)^{\alpha-1}\ri)
\\
&=&
\E\left(\e^{-(\alpha-1)X_u}\,\le(\int_u^t\e^{-(X_v-X_u)}\,dv\right)^{\alpha-1}\ri)
\\
&=&
\e^{-u\Phi(\alpha-1)}\E\left(\left(\int_0^{t-u}\e^{-(X_{v+u}-X_u)}\,dv\right)^{\alpha-1}\ri)
\\
&=&
\e^{-u\Phi(\alpha-1)}\E\left(\left(\int_0^{t-u}\e^{-X_{v}}\,dv\right)^{\alpha-1}\ri)
\\
&=&
\e^{-u\Phi(\alpha-1)}\, m^{(\alpha-1)}_{t-u}.
\end{eqnarray*}
Subsituting this expression  into  (\ref{recur3}) proves the claim. 
\end{proof}

For positive  integer values on $\alpha$ the recursive equation  (\ref{recur00})  can be solved explicitly to obtain the formula (\ref{pm22}) in the next proposition. However, we offer another proof highlighting  the  symmetry properties present in the expressions of the moments of the exponential functional. 
 
\begin{proposition} \label{productprop}
For $0\leq s\leq t\leq\infty$ and $n=1,2,\dots$ it holds
\begin{eqnarray}
\label{pm22}
&&\quad m^{(n)}_{s,t}
=n! \int_s^{t}dt_1
\int_{t_1}^{t}dt_2 \cdots \\
&&\hskip 1cm
\nonumber
\cdots
\int_{t_{n-1}}^{t}dt_n  \exp\left(-\sum_{k=1}^{n}\left(\Phi(t_{k};n-k+1)-\Phi(t_{k};n-k)\right)\right).
\end{eqnarray}
In particular, $ m^{(n)}_{s,\infty}<\infty$ if and only if the multiple integral on the right hand side of (\ref{pm22}) is finite.
\end{proposition}

\begin{proof} Let $t<\infty$ and consider 
\begin{eqnarray*}
\label{pm2}
\quad m^{(n)}_{s,t}
&=&\E\left(\left(\int_s^t\e^{-X_u}\,du\right)^n\ri)
\\
&=&\E\left(\int_s^t\cdots \int_s^t  \e^{-X_{t_1}-\dots-X_{t_n}}dt_1\dots dt_n\ri)
\\
&=&n!\,\E\left(\int_s^tdt_1\,\e^{-X_{t_1}}\,\int_{t_1}^tdt_2\,\e^{-X_{t_2}}\dots\int_{t_{n-1}}^tdt_n\,\e^{-X_{t_n}}
\right)
\\
&=&n!\int_s^t\,dt_1\int_{t_1}^t\,dt_2\dots\int_{t_{n-1}}^t\,dt_n\E\left(\e^{-(X_{t_1}+\dots+X_{t_n})} 
\right),
\end{eqnarray*}
where, in the third step, we use that
$
(t_{1},t_{2},\cdots,t_{n})\mapsto \e^{-(X_{t_1}+\dots+X_{t_n})}
$
is symmetric. By the independence of the increments 
$$
\E\left(\e^{-\alpha X_t}\ri)=\E\left(\e^{-\alpha (X_t-X_s)-\alpha
X_s} \ri)=\E\left(\e^{-\alpha (X_t-X_s)}\ri)\E\left(\e^{-\alpha
X_s} \ri).
$$ 
Consequently,
$$
\E\left(\e^{-\alpha (X_t-X_s)}\ri)=\E\left(\e^{-\alpha X_t}\ri)/\E\left(\e^{-\alpha
X_s} \ri)=\e^{-(\Phi(t;\alpha)-\Phi(s;\alpha))}.
$$ 
Since, 
$$
X_{t_1}+\dots+X_{t_n}= \sum_{k=1}^n
(n-k+1)\le(X_{t_k}-X_{t_{k-1}}\ri),\quad t_0:=0,
$$
we have 
\begin{eqnarray*}
\label{pm21}
&&\quad m^{(n)}_{s,t}
=n!\int_s^t\,dt_1\int_{t_1}^t\,dt_2\dots
\\
&&\hskip 1cm
\nonumber
\cdots
\int_{t_{n-1}}^{t}dt_n  \exp\left(-\sum_{k=1}^{n}
\le(\Phi(t_k;n-k+1)-\Phi(t_{k-1};n-k+1)\ri)\ri).
\end{eqnarray*}
Applying the initial values given in (\ref{piilap11}) yields the claimed formula (\ref{pm22}). The statement concerning  the finiteness of $m_{s,\infty}^{(n)}$ follows by applying  the monotone convergence theorem as $t\to\infty$ on both sides of (\ref{pm22}).  
\end{proof}



\begin{corollary} \label{corrmom}
Variable $I_\infty$ has all the positive moments if for all $n$
\begin{eqnarray}
\label{finmom}
\int_0^{\infty} {\rm   e}^{-(\Phi(s;n)-\Phi(s;n-1))}\, ds <\infty.
\end{eqnarray}
\end{corollary} 
\begin{proof}
 From (\ref{pm22}) we have  
\begin{eqnarray}
\label{finmom0}
m^{(n)}_t\leq n! \prod_{k=1}^n\int_0^{\infty} {\rm   e}^{-(\Phi(s;n)-\Phi(s;n-1))}\, ds.
\end{eqnarray}
The right hand side of (\ref{finmom0})  is finite if (\ref{finmom}) holds. 
Let $t\to\infty$ in (\ref{finmom0}). By monotone convergence, $m^{(n)}_\infty=\lim_{t\to\infty}m^{(n)}_t,$ and the claim is proved. 
\end{proof}

Formula (\ref{by2}) below extends the corresponding formula for subordinators found \cite{U}, see also  \cite{BY}, p.195, for 
L\'evy processes satisfying (A).  It is a straightforward implification of Proposition  \ref{productprop}.

\begin{corollary} 
\label{cor32}
Let $(X_t)_{t\geq 0}$ be a L\'evy process with the Laplace exponent
as in (\ref{piilap11}) and define  $n^*:=\min\{n\in\{1,2,\dots\}\,:\, \Phi(n)\leq 0\}.$ 
Then
\begin{equation}
\label{by2}
m^{(n)}_\infty:={\bf E}(I_{\infty}^n)=
\begin{cases}
{\displaystyle \frac{n!}{\prod _{k=1}^n \Phi(k)}}, &\mbox{if }\ n <n^*,
\\
+\infty, &\mbox{if }\ n \geq n^*. 
\end{cases}
\end{equation}.
\end{corollary}

\begin{example} 
\label{DuYo}
{\rm A much studied functional is obtained when taking $X=(X_t)_{t\geq 0}$ with $X_t=\sigma W_t+\mu t,\, \sigma >0, \mu>0,$  where $(W_t)_{t\geq 0}$ is a standard Brownian motion. In   Dufresne \cite{D} and
Yor \cite{Y} (see also Salminen and Yor \cite{SY}) it is proved that}
\begin{equation}
\label{d-y}
I_\infty:=\int_0^\infty  \exp(-(\sigma W_s+\mu s))\, 
ds
\quad\sim\quad 
H_0^{(\delta)},
\end{equation}
{\rm where $H_0^\delta$ is the first hitting time of 0 for a Bessel process of dimension
$\delta=2(1-(\mu/\sigma^2))$ starting from $\sigma/2,$  and $\sim$ means "is identical in law with". In particular, 
it holds}
\begin{equation}
\label{d-y2}
\int_0^\infty  \exp(-(2W_s+\mu s))\,
ds
\quad
\sim
\quad
\frac{1}{2\,Z_\mu},
\end{equation}
{\rm where $Z_\mu$ is a gamma-distributed random variable with rate 1 and
shape $\mu/2.$ We refer to \cite{D} for a discussion
showing how the functional on the left hand side of (\ref{d-y}) arises
as the present value of a perpetuity in a discrete model after a limiting procedure. 
Since the L\'evy exponent in this case is }
$$
\Phi(\lambda)=\lambda \mu-\frac 12\lambda^2\sigma^2,
$$
{\rm the criterium in Corollary \ref{cor32} yields}
$$
\E\left(I_\infty^n\right)<\infty\quad \Leftrightarrow\quad n<2\mu/\sigma^2,
$$ 
{\rm which readily can also be checked from (\ref{d-y2}). }
\end{example}


\end{section}
\begin{section}{First hit processes of one-dimensional diffusions}
We recall first some facts concerning the first hitting times of one-dimensional (or linear)
diffusions. Let now  $Y=(Y_s)_{s\geq 0}$ be a linear diffusion
taking values in an interval $I.$ 
To fix ideas assume that $I$ equals $\R$ or
$(0,\infty)$ or $[0,\infty)$ and that
\begin{equation}
\label{280}
\limsup_{s\to\infty}Y_s=+\infty\quad {\rm a.s.} 
\end{equation}
Assume  $Y_0=v$ and consider for $a\geq v$  the first hitting time
$$
H_a:=\inf\{s\,:\, Y_s>a\}.
$$
Defining $X_t:=H_{t+v},\, t\geq 0,$ it is easily seen -- since $Y$ is a strong Markov process -- that
$X=(X_t)_{t\geq 0}$ is an increasing purely discontinuous additive
process starting from 0. Moreover, from \eqref{280} it follows 
that $X_t<\infty$ a.s. for all $t.$ The process $X$ satisfies (A) and it holds 
\begin{equation}
\label{290}
{\bf E}_v({\rm e}^{-\beta X_t})={\bf E}_v({\rm e}^{-\beta H_{t+v}})=
\frac{\psi_{\beta}(v)}{\psi_{\beta}(t+v)},\quad t\geq 0,
\end{equation}
where $\beta\geq 0,$ ${\bf E}_v$ is the expectation associated with  $Y,\, Y_0=v,$
and $\psi_{\beta}$ is a unique (up to a multiple) positive and increasing
solution (satisfying appropriate boundary conditions) of the ODE
\begin{equation}
\label{2901}
(Gf)(x)=\beta f(x),
\end{equation}
where  $G$ denotes the differential operator associated with $Y. $ For details about diffusions
  (and further references), see It\^o and McKean \cite{IMcK}, and  \cite{BS}. 
The  Laplace transform of
$X_t$ can also be represented as follows
\begin{equation}
\label{291}
{\bf E}_v({\rm e}^{-\beta X_t})=
\exp \left( -\int_v^{t+v}
S(du)\int_0^{\infty}(1-{\rm e}^{-\beta x}) n(u,dx)\right),
\end{equation}
where  $S$ is the scale function, and $n$ is a kernel such that
for all $v\in I$ and $t\geq 0$
$$
\int_v^{t+v}\int_0^\infty (1\wedge x)n(u,dx)S(du)<\infty.
$$ 
\noindent
Representation \eqref{291} clearly reveals the structure of $X$ as a
process with independent increments. From \eqref{290} and \eqref{291}
we may conclude that 
\begin{equation}
\label{30}
\int_0^{\infty}(1-{\rm e}^{-\beta\,x})n(u,dx)=  \lim _{w\rightarrow u-}
\frac{1- {\bf E}_w({\rm e}^{-\beta X_u})}{S(u)-S(w)}.
\end{equation}

We now  pass to present examples of  exponential functionals of   
first hit processes. Firstly, we study Bessel
processes satisfying \eqref{280} and show, in particular, that the exponential
functional of the first hit process has all the moments. In our second
example it is seen that  the exponential
functional of the first hit process of geometric
Brownian motion has only finitely many moments depending on the values
of the parameters.
\begin{example} {\bf Bessel processes.}
\rm\  Let $Y$ be a Bessel process starting from $v>0$.  The
differential operator associated with $Y$ is given by
$$(Gf)(x) = \frac{1}{2}f''(x) + \frac{\delta-1}{2x} f'(x),\quad x>0,$$
where $\delta\in\R$ is called the dimension parameter. From
\cite{BS} we extract the following information
\begin{description}
\item{$\bullet$}\hskip2.5mm for $\delta\geq 2$ the boundary point $0$ is entrance-not-exit and
\eqref{280} holds,
\item{$\bullet$}\hskip2.5mm  for $0<\delta<2$  the boundary point $0$ is non-singular
and \eqref{280} holds when the boundary condition at $0$ is reflection,
\item{$\bullet$}\hskip2.5mm for $\delta\leq 0$ \eqref{280} does not hold.
\end{description}
 In case when \eqref{280} is valid
the Laplace exponent for the first hit process $X=(X_t)_{t\geq 0}$ is given for $v>0$ and $t\geq 0$ by
\begin{equation}
\label{292}
{\bf E}_v({\rm e}^{-\beta X_t}) = \frac{\psi_{\beta}(v)}{\psi_{\beta}(t)}=
 \frac{v^{1-\delta/2}\,I_{\delta/2-1}(v\sqrt{2\beta})}{t^{1-\delta/2}\,I_{\delta/2-1}((t+v)\sqrt{2\beta})},
\end{equation}
where ${\bf E}_v$ is the expectation associated with $Y$ when started
from $v$ and  $I$ denotes the modified Bessel function of the first kind. For simplicity, we wish to study the exponential functional of $X$ when $v=0.$ To find the Laplace exponent when $v=0$ we let $v\to 0$ in (\ref{292}). For this, recall 
that for $p\not= -1,-2,\dots$
\begin{equation}
\label{293}
I_{p}(v)\simeq \frac{1}{\Gamma(p+1)}\left(\frac v2\right)^p\quad {\rm as\ }
v\to 0.
\end{equation} 
Consequently, 
\begin{eqnarray}
\label{294}
\nonumber
{\bf E}_0({\rm e}^{-\beta X_t}) &=& \lim_{v\to 0}{\bf E}_v({\rm
  e}^{-\beta X_t})
\\
\nonumber
&=&
  \frac{1}{\Gamma(\nu+1)}\left(\frac{\sqrt{2\beta}}{2}\right)^{\delta/2-1}\frac{t^{\delta/2-1}}{I_{\delta/2-1}(t\sqrt{2\beta})}
\\
\nonumber 
&=:&{\rm e}^{-\Phi(t;\beta)}.
\end{eqnarray}
The validity of \eqref{finmom}, that is, the finiteness of the
positive moments, can now be checked by exploiting the
asymptotic behaviour of $I_p$ saying that 
for all $p\in\R$ (see Abramowitz and Stegun \cite{AS}, 9.7.1 p.377) 
\begin{equation}
\label{2933}
I_{p}(t)\simeq {\rm e}^{t}/\sqrt{2\pi t}\quad {\rm as\ } t\to\infty.
\end{equation}
Indeed, for  $n=1,2,\dots$ 
 \begin{eqnarray*}
{\rm e}^{-(\Phi(t;n)-\Phi(t;n-1))}&=&\frac{ n^{\delta/2-1}}{I_{\delta/2-1}(t\sqrt{2n})}\frac{I_{\delta/2-1}(t\sqrt{2(n-1)})}{(n-1)^{\delta/2-1}}
\\
&\simeq&
\left(\frac{n}{n-1}\right)^{\delta/2-1}\left(\frac{n}{n-1}\right)^{1/4}\,{\rm e}^{-t(\sqrt{2n}-\sqrt{2(n-1)})},
\end{eqnarray*}
which clearly is integrable at $+\infty.$ Consequently, by Corollary
\ref{corrmom},   the integral
functional 
$$
\int_0^\infty {\rm e}^{-X_t}\, dt
$$
has all the (positive) moments. 
\end{example}

\begin{example} {\bf Geometric Brownian motion.}\rm \
Let $Y=(Y_s)_{s\geq 0}$ be a geometric Brownian motion with parameters
 $\sigma^2>0$ and $\mu\in\R$, i.e.,
$$Y_s= \exp\big( \sigma W_s+(\mu-\frac12 \sigma^2)s\big)$$
where  $W=(W_s)_{s\geq 0}$ is a standard Brownian motion initiated at
0. Since $W_s/s\to 0$ a.s. when $s\to\infty$ it follows 
\begin{description}
\item{$\bullet$}\hskip2.5mm    $\lim_{s\to\infty}Y_s=+\infty$ a.s if $\mu>\frac12 \sigma^2,$
\item{$\bullet$}\hskip2.5mm   $\lim_{s\to\infty}Y_s=0$ a.s if $\mu<\frac12 \sigma^2.$
\item{$\bullet$}\hskip2.5mm   $\limsup_{s\to\infty}Y_s=+\infty$ and
$\liminf_{s\to\infty}Y_s=0$ a.s. if
$\mu=\frac12 \sigma^2.$  
\end{description}
 Consequently, condition
\eqref{280} is valid if and only if $\mu\geq\frac12 \sigma^2.$
Since $Y_0=1$ we consider the first hitting times of points  $a\geq 1.$ Consider 
 \begin{eqnarray*}
H_a&:=& \inf\{s\,:\, Y_s=a\}
\\
&=& \inf\big\{s\,:\, \exp\big( \sigma W_s+(\mu-\frac12 \sigma^2)s\big)=a\big\}
\\
&=& \inf\big\{s\,:\,  W_s+\frac{\mu-\frac12 \sigma^2}{\sigma}\,s=\frac{1}{\sigma}\log a\big\}.
\end{eqnarray*}
We assume now that  $\sigma>0$ and $\mu\geq\frac12 \sigma^2.$  Let
$\nu:=(\mu-\frac12 \sigma^2)/\sigma.$ Then $H_a$  is identical
in law with the first hitting time of $\log a/\sigma$ for Brownian
motion with drift $\nu\geq 0$ starting from 0. Consequently, letting $X_t:=H_{1+t}$ we have for
$t\geq 0$
 \begin{eqnarray}
\label{295}
\nonumber
{\bf E}_1({\rm e}^{-\beta X_t}) &=&
\frac{\psi_{\beta}(0)}{\psi_{\beta}({\log(1+t)}/{\sigma})}
\\
\nonumber 
&=&
\exp\left( -\left(\sqrt{2\beta +\nu^2}-\nu\right)\frac{\log(1+t)}{\sigma}\right) 
\\
 &=& (1+t)^{ -\left(\sqrt{2\beta +\nu^2}-\nu\right)/\sigma},
\\
\nonumber 
 &=:& \exp\left( -\Phi(t;\beta)\right).
\end{eqnarray}
where ${\bf E}_1$ is the expectation associated with $Y$ when started
from 1 and 
$$
\psi_{\beta}(x)= \exp\left( \left(\sqrt{2\beta +\nu^2}-\nu\right)x\right)
$$
is the increasing fundamental solution for Brownian motion with drift
(see \cite{BS} p.132). Notice that the additive process $X$ is a deterministic time change of the first hit process of Brownian motion with drift, which is a subordinator. 
We use now  Proposition  \ref{productprop} to study the moments of the perpetual integral
functional
$$
 I_\infty=\int_0^\infty {\rm e}^{-X_s}\, ds.
$$
To simplify the notation (cf. (\ref{295})) introduce 
$$
\rho(\beta):=\big(\sqrt{2\beta +\nu^2}-\nu\big)/\sigma.
$$
By formula \eqref{pm22} the $n$th moment is given by 
\begin{eqnarray}
\label{299}
\nonumber 
&&\quad{\bf E_1}(I_{\infty}^n)= n! \int_0^{\infty}dt_1 \,(1+t_1)^{-(\rho(n)-\rho(n-1))}
 \int_{t_1}^{\infty}dt_2 \, (1+t_2)^{-(\rho(n-1)-\rho(n-2))}
 \\
&&\hskip 4cm
\nonumber
\times
\int_{t_2}^{\infty}dt_3 \cdots\int_{t_{n-1}}^{\infty}dt_n\,  (1+t_n)^{-\rho(1)}
\\
\nonumber
&&\hskip1.8cm=
\begin{cases}
{\displaystyle \frac{n!}{\prod_{k=1}^n(\rho(k)-k)}}, &\mbox{if }\ n <n^*,
\\
&\\
+\infty, &\mbox{if }\ n \geq n^*, 
\end{cases}
\end{eqnarray}
where   
$$
n^*:=\min\{n\in\{1,2,\dots\}\,:\, \rho(n)-n\leq 0\}.
$$ 
Condition (\ref{finmom}) in  Corollary
\ref{corrmom} takes in this case  the form
 \begin{equation}
\label{297}
\rho(n)-\rho(n-1)>1.
\end{equation}
This being a sufficient condition for the finiteness of $m^{(n)}_\infty$ we have  
 \begin{equation}
\label{2971}
\rho(n)-\rho(n-1)>1\quad \Rightarrow \quad \rho(n)- n>0.
\end{equation}
Consider now the case $\nu=0.$ Then
\begin{equation}
\label{302}
\rho(n)>n\quad \Leftrightarrow\quad n<\frac 2{\sigma^2},
\end{equation}
i.e.,  smaller the volatility (i.e. $\sigma$) more moments of $I_\infty$
exist, as expected. Moreover, in this case
 \begin{eqnarray}
\label{303}
\nonumber 
\rho(n)-\rho(n-1)>1\ && \Leftrightarrow\quad \sqrt{2n}+
\sqrt{2(n-1)}<\frac 2{\sigma}\\
&& \Leftrightarrow\quad 2n-1 + \sqrt{4n(n-1)}<\frac 2{\sigma^2}
 \end{eqnarray}
showing, in particular, that when  $\sigma$ is ``small''  there exist
"many" $n$ satisfying (\ref{302}) but not  (\ref{303}).

\end{example}
\end{section}

\bigskip
\noindent
{\bf Funding.} This research was partially supported by Defimath project of the Research Federation of "Math\' ematiques des Pays de la Loire", by PANORisk project "Pays de la Loire" region, and by the Magnus Ehrnrooth Foundation, Finland. 

\bibliographystyle{plain}

\end{document}